\documentclass[12pt]{amsart}%
\usepackage{amssymb}
\usepackage{amsmath}
\usepackage{amsfonts}
\usepackage{graphicx}%
\usepackage{caption}

\providecommand{\U}[1]{\protect\rule{.1in}{.1in}}
\newtheorem{theorem}{Theorem}

\begin{document}
\title[On the existence of foliations]{On the existence of foliations by solutions to the exterior Dirichlet problem
for the minimal surface equation}
\author{Ari Aiolfi \ \ \ \ Daniel Bustos \ \ \ \ Jaime Ripoll }
\date{}
\maketitle

\begin{abstract}
Given an exterior domain $\Omega$ with $C^{2,\alpha}$ boundary in
$\mathbb{R}^{n}$, $n\geq3$, we obtain a $1$-parameter family $u_{\gamma}\in
C^{\infty}\left(  \Omega\right)  $, $\left\vert \gamma\right\vert \leq\pi/2$,
of solutions of the minimal surface equation such that, if $\left\vert
\gamma\right\vert <\pi/2$, $u_{\gamma}\in C^{\infty}\left(  \Omega\right)
\cap C^{2,\alpha}\left(  \overline{\Omega}\right)  $, $u_{\gamma}%
|_{\partial\Omega}=0$ with $\max_{\partial\Omega}\left\Vert \nabla u_{\gamma
}\right\Vert =\tan\gamma$ and, if $\left\vert \gamma\right\vert =\pi/2$, the
graph of $u_{\gamma}$ is contained in a $C^{1,1}$ manifold $M_{\gamma}%
\subset\overline{\Omega}\times\mathbb{R}$ with $\partial M_{\gamma}%
=\partial\Omega$. Each of these functions is bounded and asymptotic to a
constant%
\[
c_{\gamma}=\lim_{\left\Vert x\right\Vert \rightarrow\infty}u_{\gamma}\left(
x\right)  .
\]
The mappings $\gamma\rightarrow u_{\gamma}\left(  x\right)  $ (for fixed
$x\in\Omega$) and $\gamma\rightarrow c_{\gamma}$ are strictly increasing and
bounded. The graphs of these functions foliate the open subset of
$\mathbb{R}^{n+1}$%
\[
\left\{  \left(  x,z\right)  \in\Omega\times\mathbb{R}\text{, }-u_{\pi
/2}\left(  x\right)  <z<u_{\pi/2}\left(  x\right)  \right\}  .
\]
Moreover, if $\mathbb{R}^{n}\backslash\Omega$ satisfies the interior sphere
condition of maximal radius $\rho$ and if $\partial\Omega$ is contained in a
ball of minimal radius $\varrho$, then%
\[
\left[  0,\sigma_{n}\rho\right]  \subset\left[  0,c_{\pi/2}\right]
\subset\left[  0,\sigma_{n}\varrho\right]  ,
\]
where
\[
\sigma_{n}=\int_{1}^{\infty}\frac{dt}{\sqrt{t^{2\left(  n-1\right)  }-1}}.
\]
One of the above inclusions is an equality if and only if $\rho=\varrho$,
$\Omega$ is the exterior of a ball of radius $\rho$ and the solutions are radial.

These foliations were studied by E. Kuwert in \cite{K} and our result answers
a natural question about the existence of such foliations which was not
touched in \cite{K}.

\end{abstract}

\section{Introduction}

\qquad The exterior Dirichlet problem (EDP)\ for the minimal surface equation
consists on the study of existence/nonexistence and uniqueness of solutions of
the PDE boundary problem%
\begin{equation}
\left\{
\begin{array}
[c]{l}%
\mathcal{M}\left(  u\right)  :=\operatorname{div}\left(  \frac{\nabla u}%
{\sqrt{1+\left\Vert \nabla u\right\Vert ^{2}}}\right)  =0\text{, }u\in
C^{2}\left(  \Omega\right)  \cap C^{0}\left(  \overline{\Omega}\right) \\
u|_{\partial\Omega}=\varphi
\end{array}
\right.  . \label{exDP}%
\end{equation}
where $\Omega\subset\mathbb{R}^{n}$, $n\geq2$, is an exterior domain that is,
$\Lambda:=\mathbb{R}^{n}\backslash\overline{\Omega}$ is a relatively compact
domain, and $\varphi\in C^{0}\left(  \partial\Omega\right)  $ a given
function. Additionally to existence or not of solutions of (\ref{exDP}), one
is also interested on global properties of their graphs in $\mathbb{R}^{n+1}.$

In $\mathbb{R}^{2}$ the EDP has a history which goes back to J. C. C. Nitsche
who proved (Section 4 of \cite{N}) that any solution of (\ref{exDP}) has a
$C^{1}$ expansion,$\ $for $\left\Vert x\right\Vert $ big enough, of the form%
\begin{equation}
u\left(  x_{1},x_{2}\right)  =c_{1}x_{1}+c_{2}x_{2}+c\log\left\Vert
x\right\Vert +O\left(  \left\Vert x\right\Vert ^{-1}\right)  . \label{expan}%
\end{equation}

Regarding the existence/non existence problem, R. Osserman \cite{O} proved
that there is a boundary data on the disk for which the EDP (\ref{exDP}) on
the complement of the disk has no bounded solution. R. Krust \cite{Kr} proved
that Osserman's boundary data has no solution with horizontal end, that is,
$c_{1}=c_{2}=0$ in (\ref{expan}) or, equivalently, having vertical Gauss map
at infinity, leaving opened the question about the existence or not of a
boundary data for which the EDP has no solution at all that is, with no end
type restriction. This was solved by N. Kutev and F. Tomi \cite{KT} who proved
the existence of a boundary data, with arbitrarily small oscillation and with
bounded $C^{0,1}$ norm, for which (\ref{exDP}) has no solution, irrespective
of the asymptotic behavior. As to the existence problem, it is proved in
\cite{KT} and \cite{RT} that (\ref{exDP}) has\ a solution with horizontal end
under conditions involving the curvature of the boundary of the domain, the
Lipschitz constant and the oscillation of the boundary data.

Regarding the behavior in $\mathbb{R}^{n+1},$ $n\geq2,$ of the graphs of the
solutions of (\ref{exDP}), we remark that the fundamental solutions (see next
section) on the exterior of any given open ball $B$ of $\mathbb{R}^{n}$,
provide examples of foliations with horizontal ends of the open subset of
$\mathbb{R}^{n}$%
\[
\left\{  \left(  x,z\right)  \in\mathbb{R}^{n}\backslash\overline{B}%
\times\mathbb{R}\text{ such that }-v\left(  x\right)  <z<v\left(  x\right)
\right\}  ,
\]
where the graph of $v$ is the top of a generalized catenoid with neck size
determined by $B.$ This foliation is parametrized by the angle that the Gauss
map of the graph of the fundamental solution at the boundary of the domain
makes with the positive vertical axis (note that if $\gamma$ is such angle
relatively to a fundamental solution $u\in C^{2}\left(  \mathbb{R}%
^{n}\backslash B\right)  $, then $\tan\gamma=\sup_{\partial B}\left\Vert
\nabla u\right\Vert $). A question that arises is if such a similar phenomenon
happens with an arbitrary exterior domain.

This question was partially answered by the third author in $\mathbb{R}^{2}$
(Theorem 1 of \cite{R})$.$ A complete answer in the two dimensional case was
obtained in \cite{RT} where the authors prove that the limit of the leaves in
Theorem 1 of \cite{R} can be included in the foliation.

We recall that R. Krust proved in \cite{Kr} that if there are two different
solutions in $\mathbb{R}^{3}$ with the same Gauss map at infinity then there
is a continuum of solutions foliating the space in between$.$

The case $\mathbb{R}^{n}$ for $n\geq3,$ to the authors' knowledge, was
investigated only in the work of E. Kuwert \cite{K} where it is proved that
the Krust foliation theorem \cite{Kr} is true in any dimension, leaving opened
however the problem of existence or not of such foliations.

In the present paper we investigate the existence of foliations to the EDP in
$\mathbb{R}^{n}$ for $n\geq3$ in arbitrary exterior domains of $\mathbb{R}%
^{n}$ but in the special case that the boundary data $\varphi$ in (\ref{exDP})
is zero. We use in part the technique of \cite{R} for proving that an exterior
domain $\Omega$ of $C^{2,\alpha}$ class in $\mathbb{R}^{n},$ $n\geq3,$
determines a non trivial foliation of minimal hypersurfaces in $\Omega\times\mathbb{R\subset R}^{n+1}$ containing the trivial solution as a
leaf. As it happens in the $2-$dimensional case, this foliation has horizontal
ends and is parametrized by the maximal angle that the Gauss map of the leaves
in $\mathbb{R}^{n+1}$ make with the positive vertical axis at $\partial
\Omega.$ Moreover, any leaf has a limit height at infinity which can be
estimated by the geometry of the domain (see Theorem 1 for a precise statement).

A natural problem is to extend our result to more general boundary data. To
succeed, applying the technique used here (or of \cite{RT}), one needs to
guarantee the existence of at least one solution with the given boundary data.
However, although not having a counter example, we do believe, as it happens
in the $2-$dimensional case, that without hypothesis on the boundary data such
a solution may not exist. And even if one solution exists, it can possibly be
the only one. This happens in the $2-$dimensional case on the exterior of a
disk for certain boundary data, as proved in Theorem 2.9 of \cite{RT}. Even
though, it seems to us that a more difficult part on the nonzero boundary data
case is to estimate the values at infinity of the solutions: as done here, one
needs the fundamental solutions as barriers and the way they are used applies,
in principle, only for zero boundary data.

\section{Fundamental solutions}

Given $\lambda>0$ and $p\in\mathbb{R}^{n}$ let $B_{\lambda}\left(  p\right)  $
be the ball centered at $p$ and with radius $\lambda,$ $n\geq2$. The radial
function
\begin{equation}
v_{\lambda}\left(  x\right)  =\lambda\int_{1}^{\frac{r}{\lambda}}\frac
{dt}{\sqrt{t^{2\left(  n-1\right)  }-1}}\text{, }r=\left\Vert x-p\right\Vert
\text{, }x\in\mathbb{R}^{n}\backslash B_{\lambda}\left(  p\right)  ,
\label{ncat}%
\end{equation}
is a solution of (\ref{exDP}) in $\mathbb{R}^{n}\backslash B_{\lambda}\left(
p\right)  $ vanishing at $\partial B_{\lambda}\left(  p\right)  .$ We call
$v_{\lambda}$, or any vertical translation of $v_{\lambda},$ a fundamental
solution. The graph of $v_{\lambda}$ is half of a $n-$dimensional catenoid. By
using isometries and homotheties one obtains a family of radial solutions,
which we also call fundamental solutions, defined in the exterior of any fixed
ball which gradient at the boundary of the ball varies from $0$ to $\infty.$

In this paper we are interested only when $n\geq3.$ In this case we then have%
\begin{equation}
0<\sigma_{n}:=\int_{1}^{\infty}\frac{dt}{\sqrt{t^{2\left(  n-1\right)  }-1}%
}<\infty\label{sig}%
\end{equation}
so that, from (\ref{ncat}), $v_{\lambda}\left(  x\right)  $ has a limit as
$\left\Vert x\right\Vert \rightarrow\infty$ not depending on $p,$ which we
denote by $v_{\lambda}\left(  \infty\right)  $ and which is given by
\begin{equation}
v_{\lambda}\left(  \infty\right)  =\sigma_{n}\lambda. \label{vlinf}%
\end{equation}

\begin{figure}[h]
\centering
\includegraphics[width=1\linewidth]{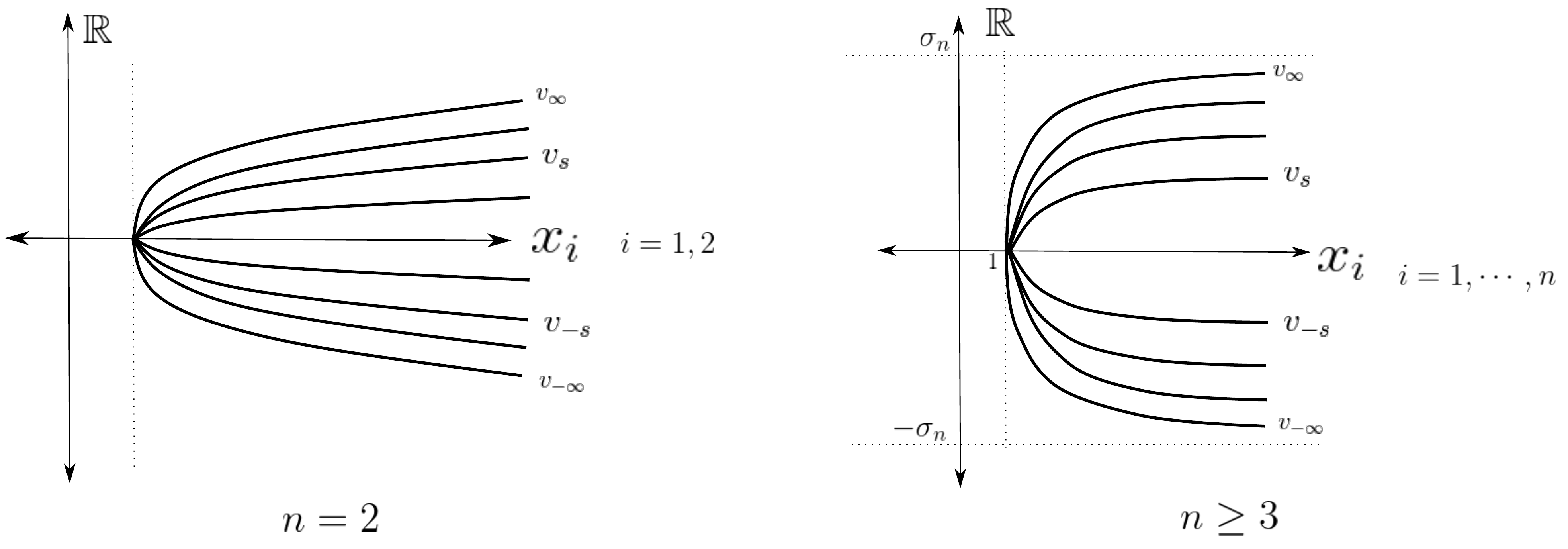} 
\caption*{Fundamental Solutions with $\lambda=1$}
\label{fig:fig1}%
\end{figure}

\section{The result and its proof}

A fundamental tool in PDE, used several times in the proof of Theorem
\ref{mt}, is the comparison principle. In our case it states that if $\Omega$
is a bounded domain in $\mathbb{R}^{n}$, $u,v\in C^{2}\left(  \Omega\right)  $
satisfies $\mathcal{M}\left(  u\right)  =\mathcal{M}\left(  v\right)  =0$ and
$u\leq v$ at $\partial\Omega$ that is,%
\[
\lim\sup_{k}\left(  u(x_{k})-v(x_{k})\right)  \leq0
\]
for any sequence $x_{k}$ in $\Omega$ which leaves any compact subset of
$\Omega,$ then $u\leq v$ in $\Omega$ (Proposition 3.1 of \cite{RT2})$.$ An
easy consequence of the comparison principle is the maximum principle which
asserts that if $u,v\in C^{2}\left(  \Omega\right)  \cap C^{0}\left(
\overline{\Omega}\right)  $ satisfy $\mathcal{M}\left(  u\right)
=\mathcal{M}\left(  v\right)  =0$ in $\Omega$ then
\[
\max_{\overline{\Omega}}\left\vert u-v\right\vert =\max_{\partial\Omega
}\left\vert u-v\right\vert
\]
(Proposition 3.2 of \cite{RT2})$.$

The maximum principle has an useful application on Differential Geometry,
known as the tangency principle. In our case it says that if $M_{1}$ and
$M_{2}$ are minimal hypersurfaces of $\mathbb{R}^{n}$ (with or without
boundary and not necessarily graphs) that have a tangency at some interior or
boundary point $p\in M_{1}\cap M_{2},$ and if $M_{1}$ is in one side of
$M_{2}$ in a neighborhood $p,$ then $M_{1}$ coincides with $M_{2}$ in a
neighborhood of $p$ \cite{FS}.

We also remark that once we have a priori $C^{1}$ estimates for the solutions
of the minimal surface equation (or to more general quasi linear elliptic
PDEs) we also have $C^{1,\alpha}$ a priori estimates from the H\"{o}lder
theory (Ch 13 of \cite{GT}). Then well known arguments (see, for example,
Section 2.1 of \cite{RT2}) allow to reduce the $C^{2,\alpha}$ a priori
estimates and the regularity of solutions of quasi linear elliptic PDEs to a
priori estimates and regularity theory of linear elliptic PDEs (Ch 6 of
\cite{GT}).

In the statement of Theorem \ref{mt} we set, for convenience, $s:=\tan\gamma$,
$\left\vert \gamma\right\vert \leq\pi/2$ and, we use $u_{s}$ and $u_{s}\left(
\infty\right)  $ instead of $u_{\gamma}$ and $c_{\gamma}$ as in the Abstract.

\begin{theorem}
\label{mt}Assume that $\Omega$ is an exterior domain of $C^{2,\alpha}$ class
such that $\Lambda:=\mathbb{R}^{n}\backslash\overline{\Omega},$ $n\geq3,$
satisfies the interior sphere condition with maximal radius $\rho$, namely:
Given $p\in\partial\Lambda$, there is a $\left(  n-1\right)  -$dimensional
sphere $S_{p}$ of radius $\rho$ such that $p\in S_{p}$, $S_{p}\subset
\overline{\Lambda}$ and $\rho$ is maximal under these conditions. Let
$\varrho$ be the radius of the smallest open ball $B_{\varrho}$ of
$\mathbb{R}^{n}$ such that $\partial\Omega\subset\overline{B}_{\varrho}.$
Given $s\in\left[  -\infty,\infty\right]  $ there is a bounded function
$u_{s}\in C^{\infty}\left(  \Omega\right)  $ satisfying $\mathcal{M}\left(
u_{s}\right)  =0$ in $\Omega,$ $u_{-s}=-u_{s}$ and such that:

If $-\infty<s<\infty$ then $u_{s}\in C^{\infty}\left(  \Omega\right)  \cap
C^{2,\alpha}\left(  \overline{\Omega}\right)  ,$
\begin{equation}
u_{s}|_{\partial\Omega}=0 \label{bou}%
\end{equation}
and%
\begin{equation}
\max_{\partial\Omega}\left\Vert \nabla u_{s}\right\Vert =\max_{\Omega
}\left\Vert \nabla u_{s}\right\Vert =\left\vert s\right\vert . \label{gr}%
\end{equation}

The graph of $u_{\infty}$ is contained in a $C^{1,1}$-manifold $M\subset
\overline{\Omega}\times\mathbb{R}$ with boundary $\partial M=$ $\partial
\Omega.$

For any $s\in\left[  -\infty,\infty\right]  $ there exists the limit
\begin{equation}
u_{s}\left(  \infty\right)  :=\lim_{\left\Vert x\right\Vert \rightarrow\infty
}u_{s}\left(  x\right)  , \label{cinf}%
\end{equation}
and%
\begin{equation}
\lim_{\left\Vert x\right\Vert \rightarrow\infty}\left\Vert \nabla u_{s}\left(
x\right)  \right\Vert =0. \label{gauss}%
\end{equation}

Moreover, the maps $s\mapsto u_{s}\left(  x\right)  $, for fixed $x\in\Omega$
and $s\mapsto u_{s}\left(  \infty\right)  $ are strictly increasing, bounded
and we have the inclusions
\begin{equation}
\left[  -\sigma_{n}\rho,\sigma_{n}\rho\right]  \subset\left[  -u_{\infty
}\left(  \infty\right)  ,u_{\infty}\left(  \infty\right)  \right]
\subset\left[  -\sigma_{n}\varrho,\sigma_{n}\varrho\right]  \label{inc}%
\end{equation}
where $\sigma_{n}$ is given by (\ref{sig}). If one of the inclusions is an
equality then $\rho=\varrho$, $\Omega$ is the exterior of a ball of radius
$\rho$ and the $u_{s}$ are the fundamental solutions.

Finally, the graphs of the solutions of $u_{s},$ $s\in\left(  -\infty
,\infty\right)  $ foliate the open subset of $\mathbb{R}^{n+1}$
\begin{equation}
O:=\left\{  \left(  x,z\right)  \in\mathbb{R}^{n+1}\backslash\overline{\Omega
}\times\mathbb{R}\text{ such that }u_{-\infty}\left(  x\right)  <z<u_{\infty
}\left(  x\right)  \right\}  . \label{O}%
\end{equation}

\end{theorem}

\begin{proof}
We first consider the case that $-\infty<s<\infty$. Since the case $s=0$ is
trivial and, since if $u_{s}$ is a solution satisfying (\ref{bou}), (\ref{gr})
(\ref{cinf}) and (\ref{gauss}) then $u_{-s}=-u_{s}$ is a solution also
satisfying these conditions, we may assume $s>0$. Let $a\in\mathbb{R}$ be such
that $B_{a}=B_{a}\left(  0\right)  $, the open ball in $\mathbb{R}^{n}$ of
radius $a$ centered at origin, contains $\overline{\Lambda}$. Let $v_{a}\in
C^{0}\left(  \mathbb{R}^{n}\backslash B_{a}\right)  $ be given by (\ref{ncat})
with $p=0$. We see that $\left\Vert \nabla v_{a}\left(  x\right)  \right\Vert
\rightarrow0$ as $\left\Vert x\right\Vert \rightarrow\infty$ and we may then
choose $k\in\mathbb{N}$, $k>a+1$, large enough such that
\begin{equation}
\left\Vert \nabla v_{a}\right\Vert _{\partial B_{k}}\leq\frac{s}{2}.
\label{vas}%
\end{equation}

Set $\Omega_{k}=B_{k}\cap\Omega$ and%
\begin{equation}
T_{k}=\left\{  t\geq
0\ ;\begin{aligned} \ &\ \exists~w_{t}\in C^{2,\alpha}\left( \overline{\Omega }_{k}\right) \text{ s.t. }\mathcal{M}\left( w_{t}\right) =0,\\ &~\sup\nolimits_{\overline{\Omega}_{k}}\left\Vert \nabla w_{t}\right\Vert \leq s, \ w_{t}|_{\partial\Omega}=0,~w_{t}|_{\partial B_{k}}=t \end{aligned}\right\}
.\qquad\label{Tk}%
\end{equation}

The set $T_{k}$ is not empty since $0\in T_{k}$. Moreover, $\sup T_{k}<\infty$
since
\[
\sup_{\overline{\Omega}_{k}}\left\Vert \nabla w_{t}\right\Vert \leq s
\]
for all $t\in T_{k}$. We will prove that
\[
t_{k}:=\sup T_{k}\in T_{k}%
\]
and that
\begin{equation}
\sup_{\Omega_k}\left\Vert \nabla w_{t_{k}}\right\Vert =\sup_{\partial\Omega_k
}\left\Vert \nabla w_{t_{k}}\right\Vert =s\text{.} \label{wtk}%
\end{equation}

Taking a sequence $\left(  t_{m}^{k}\right)  $ in $T_{k}$ converging to
$t_{k}$ as $m\rightarrow\infty$ the corresponding functions $w_{t_{m}^{k}}$
have uniformly bounded $C^{1}$ norm. By elliptic PDE theory (\cite{GT},
\cite{RT2}) there is a subsequence of $w_{t_{m}^{k}}$ converging on the
$C^{2}$ norm on $\overline{\Omega}_{k}$ to a function $w_{k}\in C^{2,\alpha
}\left(  \overline{\Omega}_{k}\right)  $ which satisfies $\mathcal{M}\left(
w_{k}\right)  =0$ in $\Omega_{k}$. Clearly $w_{k}|_{\partial\Omega}=0$,
$w_{k}|_{\partial B_{k}}=t_{k}$ and $\sup_{\Omega_{k}}\left\Vert \nabla
w_{k}\right\Vert \leq s$. It follows that $t_{k}\in T_{k}$ and that
$w_{k}=w_{t_{k}}$.

From the maximality of $t_{k}$ we claim that we cannot have $\sup_{\Omega_{k}%
}\left\Vert \nabla w_{k}\right\Vert \nolinebreak<\nolinebreak s.$ Indeed:
Consider a function $\phi\in C^{2,\alpha}\left(  \mathbb{R}^{n}\right)  $ such 
that $\phi|_{B_{k-1}}=0$ and $\phi|_{\mathbb{R}^{n}\backslash B_{k}}=1$, set
\[
C_{0}^{2,\alpha}(\overline{\Omega}_{k})=\left\{  \left.  \omega\in
C^{2,\alpha}(\overline{\Omega}_{k})\text{ }\right\vert \text{ \ \ }%
\omega|_{\partial\Omega_{k}}=0\right\}  ,
\]
and define $T\colon\,[-1,1]\times C_{0}^{2,\alpha}(\overline{\Omega}%
_{k})\rightarrow C^{\alpha}(\overline{\Omega}_{k})$ by
\[
T\left(  t,\omega\right)  =\mathcal{M}\left(  \omega+w_{k}+t\phi\right)  .
\]
Then $T\left(  0,0\right)  =0.$ One may see that the Fr\'{e}chet derivative
$\partial_{2}T\left(  0,\omega_{k}\right)  =d\mathcal{M}_{w_{k}}$ is
invertible (Theorem 3.3 of \cite{GT}) so that, from the implicit function
theorem on Banach spaces (Theorem 17.6 of \cite{GT}), there exists a
continuous function $t\mapsto\omega\left(  t\right)  \in C_{0}^{2,\alpha
}(\overline{\Omega}_{k})$ (continuous on the $C^{2,\alpha}$ topology)$,$ with
$\omega(0)=0$ such that $T\left(  t,\omega(t)\right)  =0,$ $t\in\left(
-\varepsilon,\varepsilon\right)  .$ Therefore, since $\left\Vert
\operatorname{\nabla}w_{k}\right\Vert _{\Omega_k}<s$ there exists
$t\in\left(  0,\varepsilon\right)  $ such that
\[
\sup_{\Omega_k}\left\Vert \nabla\left(  \omega\left(  t\right)
+w_{k}+t\phi\right)  \right\Vert <s.
\]
Since
\[
\mathcal{M}\left(  \omega\left(  t\right)  +w_{k}+t\phi\right)  =T(t,\omega
(t))=0,
\]
$\omega\left(  t\right)  +w_{k}+t\phi=0$ at $\partial\Omega$ and
$\omega(t)+w_{k}+t\phi=t_{k}+t$ at $\partial B_{k},$ it follows that
$t_{k}+t\in T_{k},$ contradiction since $t_{k}=\sup T_{k}.$ We then have
$\sup_{\Omega_{k}}\left\Vert \nabla w_{k}\right\Vert =s$.
We claim that%
\begin{equation}
\sup_{\partial B_{k}}\left\Vert \nabla w_{k}\right\Vert \leq s/2. \label{was}%
\end{equation}

Indeed: Since the graph of $v_{a}$ is vertical at $\partial B_{a}$ it follows
from the comparison principle (see \cite{GT}, Ch 10, or Proposition 3.1 of
\cite{RT2}) that
\begin{equation}
v_{a}+t_{k}-v_{a}(x_{0})\leq w_{k}\leq t_{k} \label{in}%
\end{equation}
where $x_{0}\ $is any but fixed point of $\partial B_{k}.$ From (\ref{vas})
and (\ref{in}) we get (\ref{was}). By the gradient maximum principle
(\cite{GT}, Ch 15) we obtain%
\[
\sup_{\Omega_k}\left\Vert \nabla w_{k}\right\Vert =\sup_{\partial\Omega
_k}\left\Vert \nabla w_{k}\right\Vert =s.
\]

Letting $k\rightarrow\infty$ and using the diagonal method we obtain a
subsequence of $w_{k}$ converging uniformly $C^{2}$ on compact subsets of
$\overline{\Omega}$ to a function $u_{s}\in C^{2,\alpha}\left(  \overline
{\Omega}\right)  $ satisfying $\mathcal{M}\left(  u_{s}\right)  =0$ in
$\Omega,$ (\ref{bou}) and (\ref{gr}). From elliptic PDE regularity \cite{GT}
$u_{s}\in C^{\infty}\left(  \Omega\right)  $.

Now, for any $s\in\left[  0,\infty\right)  ,$ the graph $G_{s}$ of $u_{s}$ is
by construction of (uniform) bounded slope (see \cite{S}). It follows from
Proposition 3 of \cite{S} that $G_{s}$ is \emph{regular at infinity }that is,
$u_{s}$ has a twice differentiable expansion%
\begin{equation}
u_{s}\left(  x\right)  =c_{s}+a_{s}\left\Vert x\right\Vert ^{2-n}+\sum
_{j=1}^{n}c_{s,j}x_{j}\left\Vert x\right\Vert ^{-n}+O\left(  \left\Vert
x\right\Vert ^{-n}\right)  \label{exp}%
\end{equation}
from which it follows that%
\begin{equation}
u_{s}\left(  \infty\right)  :=\lim_{\left\Vert x\right\Vert \rightarrow\infty
}u_{s}\left(  x\right)  =c_{s}. \label{cs}%
\end{equation}

It also follows from (\ref{exp}) that
\[
\lim_{\left\Vert x\right\Vert \rightarrow\infty}\left\Vert \nabla
u_{s}\right\Vert \left(  x\right)  =0
\]
which implies that $G_{s}$ is horizontal at infinity that is, (\ref{gauss}) is
satisfied. This proves that (\ref{cinf}) and (\ref{gauss}) are satisfied for
$s\in\left[  0,\infty\right)  .$

Let $v_{\varrho}$ be the fundamental solution on $\mathbb{R}^{n}\backslash
B_{\varrho}$ which gradient infinity at $\partial B_{\varrho}.$ Given
$s\in\left[  0,\infty\right)  $ we claim that $u_{s}\left(  \infty\right)
<v_{\varrho}\left(  \infty\right)  .$ Indeed, coming from $-\infty$ with the
graph $G_{\varrho}$ of $v_{\varrho}$ using vertical translations$,$ since the
gradient of $v_{\varrho}$ at the boundary of $B_{\varrho}$ is infinity, it
follows from the tangency principle that the first contact between
$G_{\varrho}$ and the graph of $u_{s}$ has to be at infinity and with the
boundary of $G_{\varrho}$ strictly below the level $x_{n+1}=0.$ Hence, at the
level $x_{n+1}=0$ one necessarily has $u_{s}\left(  \infty\right)
<v_{\varrho}\left(  \infty\right)  .$ It follows from the claim and from
(\ref{vlinf}) that $u_{s}$ is bounded by $\sigma\varrho$ for all $s\in\left[
0,\infty\right)  .$

Clearly we have $u_{s}\leq u_{t}$ and also $u_{s}\left(  \infty\right)  \leq
u_{t}\left(  \infty\right)  $ if $s\leq t$. Hence, for any increasing sequence
$s_{m}\rightarrow\infty$ the sequence $u_{s_{m}}$ converges uniformly on
compact subsets of $\Omega$ to a $C^{\infty}$ function $u_{\infty}$ in
$\Omega$ satisfying $\mathcal{M}\left(  u_{\infty}\right)  =0$.

For proving that the graph $G_{\infty}$ of $u_{\infty}$ is contained in a
$C^{1,1}$ manifold with boundary $\partial\Omega$ consider a fixed ball
$B_{a}$ with $a>\varrho.$ By \cite{Mi}, given $s\in\left[  0,\infty\right]  $
there is a minimizer $v_{s}$ on the space $\operatorname*{BV}\left(
\Omega_{a}\right)  $ of bounded variation functions on $\Omega_{a}$ (see
\cite{G})$,$ for the functional
\[
\mathcal{F}_{s}\left(  w\right)  =\int_{\Omega_{a}}\sqrt{1+\left\Vert \nabla
w\right\Vert ^{2}}+\int_{\partial\Omega_{a}}\left\vert w-\phi_{s}\right\vert
,\text{ }w\in\operatorname*{BV}\left(  \Omega_{a}\right)  ,
\]
where $\phi_{s}\in C^{\infty}\left(  \partial\Omega_{a}\right)  $ satisfies
$\phi_{s}|_{\partial\Omega}=0$, $\phi_{s}|_{\partial B_{a}}=u_{s}|_{\partial
B_{a}}.$ Since $u_{s}$ is also a minimizer for $\mathcal{F}_{s}$ for $0\leq
s<\infty,$ we have $u_{s}|_{\Omega_{a}}=v_{s}$ by uniqueness \cite{Mi} (the
equality is in\ $\operatorname*{BV}\left(  \Omega_{a}\right)  $). Noting that
\begin{align*}
\lim_{s\rightarrow\infty}\mathcal{F}_{s}\left(  w\right)   &  =\mathcal{F}%
_{\infty}\left(  w\right)  ,\text{ }w\in\operatorname*{BV}\left(  \Omega
_{a}\right)  ,\\
\lim_{s\rightarrow\infty}\mathcal{F}_{s}\left(  u_{s}\right)   &
=\lim_{s\rightarrow\infty}\mathcal{F}_{\infty}\left(  u_{s}\right)  ,
\end{align*}
we have (writing only $u_{s}$ instead of $u_{s}|_{\Omega_{a}})$%
\begin{align*}
\mathcal{F}_{\infty}\left(  v_{\infty}\right)   &  =\lim_{s\rightarrow\infty
}\mathcal{F}_{s}\left(  v_{\infty}\right)  \geq\lim_{s\rightarrow\infty
}\mathcal{F}_{s}\left(  v_{s}\right)  =\lim_{s\rightarrow\infty}%
\mathcal{F}_{s}\left(  u_{s}\right) \\
&  =\lim_{s\rightarrow\infty}\mathcal{F}_{\infty}\left(  u_{s}\right)
\geq\mathcal{F}_{\infty}\left(  u_{\infty}\right)  ,
\end{align*}
where, in the last inequality, we used that $\mathcal{F}_{\infty}$ is lower
semicontinuous. It follows that $\mathcal{F}_{\infty}\left(  v_{\infty
}\right)  =\mathcal{F}_{\infty}\left(  u_{\infty}\right)  $ and hence, by
uniqueness, $v_{\infty}=u_{\infty}$ in $\Omega_{a}$. From Theorem 4.2 of
\cite{Bo} applied to the functional $\mathcal{F}_{\infty}$, by choosing
$\Phi=\partial\Omega$, $\phi_{i}\equiv0$, and using also Theorem 4.7, we
conclude that the graph of $u_{\infty}$ is contained in a $C^{1,1}$ manifold
$M$ with boundary which boundary is $\partial\Omega.$

We have seen that $s\rightarrow u_{s}\left(  \infty\right)  $ is increasing
and bounded by $\sigma\varrho.$ If $c:=\lim_{s\rightarrow\infty}u_{s}\left(
\infty\right)  $ then we have $u_{s}\leq u_{\infty}\leq c,$ $s\in\left[
0,\infty\right)  ,$ by the comparison principle, and hence there is the limit
$u_{\infty}\left(  \infty\right)  $ of $u_{\infty}\left(  x\right)  $ as
$\left\Vert x\right\Vert \rightarrow\infty$ and $u_{\infty}\left(
\infty\right)  =c,$ proving the second inclusion of (\ref{inc}). We shall
prove now (\ref{gauss}) for $s=\infty.$

By the way $u_{\infty}$ is obtained we can not conclude directly that the
graph of $u_{\infty}$ is of (uniform) bounded slope and hence we don't know if
$u_{\infty}$ is regular at infinity and admits an expansion as (\ref{exp}).
But this is actually the case, indeed:\ Since $u_{\infty}\left(
\infty\right)  =c$ the tangent cone to the graph of $u_{\infty}$ at infinity
is the hyperplane $\mathbb{R}^{n}=\left\{  x_{n+1}=0\right\}  $ of
$\mathbb{R}^{n+1}$ (see \cite{Si}) and hence, from Theorem 1 of \cite{Si} it
follows that $\nabla u_{\infty}$ has a limit at infinity and $\left\Vert
\nabla u_{\infty}\right\Vert $ is bounded outside some compact. Since $u_{\infty
}$ is bounded this limit has to be zero and this proves (\ref{gauss}) for
$s=\infty.$

Let $c\in\lbrack0,\sigma_{n}\rho]$ be given. We prove that there is a non
negative solution $w_{c}\in C^{0}\left(  \overline{\Omega}\right)  \cap
C^{\infty}\left(  \Omega\right)  $ of (\ref{exDP}) such that $w_{c}%
|_{\partial\Omega}=0$ and%
\[
\underset{\left\Vert x\right\Vert \rightarrow\infty}{\lim}w_{c}\left(
x\right)  =c.
\]

Define
\begin{equation}
\digamma=\left\{  f\in C^{0}\left(  \overline{\Omega}\right)
;\begin{aligned} \ &~f\text{ is a subsolution of }\mathcal{M}\text{ in }\Omega,\\ &~f=0\text{ in }\partial\Omega\text{ and }\limsup \nolimits_{\left\Vert x\right\Vert \rightarrow\infty}f\left( x\right) \leq c \end{aligned}\right\}
.\qquad\label{per}%
\end{equation}

Clearly $\digamma\neq\varnothing$ and it follows from the the comparison
principle that $f\leq c$ for all $f\in\digamma.$ We may then apply Perron's
method (\cite{GT}, Section 2.8) to conclude that
\[
w_{c}\left(  x\right)  =\sup\left\{  f\left(  x\right)  ;\text{ }f\in
\digamma\right\}  \text{, }x\in\overline{\Omega},
\]
is $C^{\infty}$ and satisfies $\mathcal{M}\left(  w_{c}\right)  =0$ in
$\Omega$. For proving that
\begin{equation}
\lim_{\left\Vert x\right\Vert \rightarrow\infty}w_{c}\left(  x\right)  =c
\label{wc}%
\end{equation}
take $a>0$ large enough, such that $\overline{\Lambda}\subset B_{a}$ satisfies
$v_{a}\left(  \infty\right)  >c$. We have that $f\in C^{0}\left(
\overline{\Omega}\right)  $ given by%
\[
f\left(  x\right)  =\left\{
\begin{array}
[c]{l}%
0\text{ if }x\in\overline{\Omega}\cap B_{a}\\
\max\{0,v_{a}\left(  x\right)  -\left(  v_{a}\left(  \infty\right)  -c\right)
\}\text{, if }x\in\mathbb{R}^{n}\backslash B_{a}%
\end{array}
\right.
\]
is a subsolution relatively to the (\ref{exDP}) satisfying $f|_{\partial
\Omega}=0$ and
\begin{equation}
\underset{\left\Vert x\right\Vert \rightarrow\infty}{\lim}f\left(  x\right)
=c. \label{win}%
\end{equation}
It follows that $f\in\digamma$ and then $f\leq w_{c}\leq c,$ which proves
(\ref{wc}).

It remains to prove that $w_{c}$ extends $C^{0}$ to $\overline{\Omega}$ and
that $w_{c}|_{\partial\Omega}=0$. Given $p\in\partial\Omega$, by hypothesis
there is an open ball $B_{\rho}$ contained in $\Lambda$ such that $\partial
B_{\rho}$ is tangent to $\partial\Omega$ ($=\partial\Lambda)$ at $p$. Since
\[
c\leq\sigma_{n}\rho=v_{\rho}\left(  \infty\right)
\]
and $v_{\rho}=0$ at $\partial B_{p}$ it follows from the comparison
principle\ that $0\leq w_{c}\leq v_{\rho}$. Since $p$ is arbitrary this proves
the claim that is, $w_{c}$ extends $C^{0}$ to $\overline{\Omega}$ and
$w_{c}|_{\partial\Omega}=0$.

Now, assume that $0\leq c<\sigma_{n}\rho.$ Then we may find a fundamental
solution $\widetilde{v}$ defined on the exterior of a ball of radius $\rho,$
contained in $\Lambda$, tangent to $\partial\Omega$ with bounded gradient at
the boundary of the ball and such that
\[
\widetilde{v}\left(  \infty\right)  =\frac{c+\sigma_{n}\rho}{2}.
\]

By the comparison principle it follows that $0\leq w_{c}\leq\widetilde{v}.$
This proves that $w_{c}$ extends $C^{1}$ to $\overline{\Omega}$ and, by PDE
regularity \cite{GT}, $w_{c}\nolinebreak\in\nolinebreak C^{2,\alpha}\left(
\overline{\Omega}\right)  \nolinebreak\cap\nolinebreak C^{\infty}\left(
\Omega\right)  .$ Setting%
\[
s_{c}=\max_{\partial\Omega}\left\Vert \nabla w_{c}\right\Vert ,
\]
we prove that $u_{s_{c}}=w_{c}.$ By contradiction, assume the opposite. Then,
setting%
\begin{equation}
d:=\lim_{\left\Vert x\right\Vert \rightarrow\infty}u_{s_{c}} \label{hav}%
\end{equation}
we cannot have $d>c$ or $d<c$. Indeed: Assume, by contradiction, that $d>c.$
Let $p\in\partial\Omega$ be such that $\left\Vert \nabla w_{c}\right\Vert
\left(  p\right)  =s_{c}.$ If $\left\Vert \nabla u_{s_{c}}\right\Vert \left(
p\right)  =s_{c}$ we cannot have $w_{c}\left(  x\right)  \leq u_{s_{c}}\left(
x\right)  $ for all $x\in\overline{\Omega}$ because of the boundary tangency
principle. But if have $w_{c}>u_{s_{c}}$ this inequality must hold only on a
bounded open subset of $\Omega$ since $c<d$. One can then make a vertical
translation of the graph of one of the solutions to get a tangency between
their graphs, with one of them in one side of the other, contradicting the
tangency principle.

The remaining possibility%
\[
\left\Vert \nabla u_{s_{c}}\right\Vert \left(  p\right)  <s_{c}=\left\Vert
\nabla w_{c}\right\Vert \left(  p\right)
\]
also implies that $w_{c}>u_{s_{c}}$ must hold on a bounded open subset of
$\Omega$ leading, as before, to a contradiction with the tangency principle.
The case that $d<c$ cannot happen by the same arguments. This proves that
$c=d$ and, arguing with the tangency principle again, that $w_{c}=u_{s_{c}}.$

Finally, take an increase sequence $c_{m}\in\left[  0,\sigma_{n}\rho\right)  $
converging to\ $\sigma_{n}\rho$ as $m\rightarrow\infty$. The sequence
$s_{c_{m}}$ is increasing and then has a limit $s\in\left[  0,\infty\right]
.$ The sequence $\left(  u_{s_{c_{m}}}\right)  $ converges uniformly $C^{2}$
on compact subsets of $\Omega$ to a solution $u_{s}\in C^{0}\left(
\overline{\Omega}\right)  \cap C^{\infty}\left(  \Omega\right)  ,$
$u_{s}|_{\partial\Omega}=0$ and $\sup_{\partial\Omega}\left\Vert \nabla
u_{s}\right\Vert =s.$ As before we obtain $u_{s}=w_{\sigma_{n}\rho}$, proving
that
\[
\left[  0,\sigma_{n}\rho\right]  \subset\left[  0,u_{\infty}\left(
\infty\right)  \right]  .
\]
This concludes the proof of (\ref{inc}).

If one of the inclusions in (\ref{inc}) is an equality and the corresponding
graphs of the solutions with infinite gradient at $\partial\Omega$ are not the
same, then either one is below the other or they intersect in interior points.
The first case cannot occur because of the boundary tangency principle. The
second case neither because otherwise one can make a vertical translation of
one of them to get a tangency between the graphs, with one in one side of the
other, contradicting the tangency principle. Hence, in case of equality in
some of the inclusions (\ref{inc}), $\Omega$ is the exterior of a ball of
radius $\rho=\varrho.$ Is a particular consequence of the proof of the
foliation property, given below, that the solutions $u_{s}$ are necessarily
the fundamental solutions.

For proving that the graphs of the solutions $u_{s},$ $s\in\left(
-\infty,\infty\right)  ,$ foliate the open subset $O$ of $\mathbb{R}^{n}$
(defined in (\ref{O})) we apply Theorem 2 of \cite{K}. It is enough to prove
that any solution $u\in C^{0}\left(  \overline{\Omega}\right)  $ of the
minimal surface equation in $\Omega$ with horizontal end and such that
$u|_{\partial\Omega}=0$ coincides with $u_{s}$ for some $s\in\left[
-\infty,\infty\right]  .$

By using Theorems 4.2 and 4.7 of \cite{Bo}, as above, we may conclude that the
graph of $u$ is a $C^{1,1}$ manifold $M$ with boundary and, since $u\in
C^{0}\left(  \overline{\Omega}\right)  $ is a solution of the minimal surface
equation in $\Omega$, $M$ is a minimal hypersurface with boundary
$\partial\Omega$ of $\mathbb{R}^{n}$. Representing $M,$ locally, as a graph
near any given point of $\partial\Omega\ (=\partial M),$ we may use PDE
regularity theory to conclude that, indeed, $M$ is a $C^{2,\alpha}$ manifold.
Moreover, the assumption that $u$ has horizontal end implies, as already
argued before, that $u$ is bounded and that there exists the limit
\[
d:=\lim_{\left\Vert x\right\Vert \rightarrow\infty}u\left(  x\right)  .
\]

If $M$ has no vertical tangent space at any point of $\partial\Omega$ then it
follows by PDE regularity that $u\in C^{2,\alpha}\left(  \overline{\Omega
}\right)  \cap C^{\infty}\left(  \Omega\right)  $. Setting $s=\max
_{\partial\Omega}\left\Vert \nabla u\right\Vert ,$ we can argue as before to
prove that $u=u_{s}.$

Assume that $M$ has a vertical tangent space at some point of $\partial
\Omega.$ We claim then that $u=u_{\infty}$ or $u=u_{-\infty}.$ We first prove
that $d=u_{\infty}\left(  \infty\right)  $ or $d=u_{-\infty}\left(
\infty\right)  $. By contradiction, first assume that $0<u_{\infty}\left(
\infty\right)  <d .$

Arguing with the tangency principle it is easy to see then that $u_{\infty
}\leq u.$ But then $u_{\infty}\in C^{0}\left(  \overline{\Omega}\right)  $ and
the graph $G$ of $u_{\infty}$ is a minimal hypersuface of $C^{2,\alpha}$ class
with boundary $\partial\Omega$ which has a vertical tangent space at some
point $p\in\partial\Omega.$ The hypersurfaces $G$ and $M$ then must have a
tangency at $p.$ By the boundary tangency principle it follows that $G=M,$ contradiction!

If $0\leq d<u_{\infty}\left(  \infty\right)  $, since $u_{s}$ converges
uniformly on compacts of $\Omega$ to $u_{\infty},$ as $s\rightarrow\infty,$
there is $s$ large enough such that $u_{s}\left(  \infty\right)  >d$. By using
the tangency principle one may see that this leads to a contradiction. For
similar reasons one excludes the case $d<u_{-\infty}\left(  \infty\right)  $
and $u_{-\infty}\left(  \infty\right)  <d\leq0.$

It then follows that $d=u_{\infty}\left(  \infty\right)  $ or $d=u_{-\infty
}\left(  \infty\right)  $ from what one easily obtains, from the tangency
principle once more, that $u=u_{\infty}$ or $u=u_{-\infty}.$ This concludes
with the proof of the theorem.
\end{proof}

\noindent\textbf{Remarks.}

\noindent(a) It is true that the graph of the limit solution $u_{\infty}$ of
the EDP in $\mathbb{R}^{2}$ is a $C^{1,\alpha}$ surface with boundary.
Moreover, it holds $u_{\infty}\in C^{0}\left(  \overline{\Omega}\right)  $ in
this case \cite{RT}. In higher dimensions, as proved in the Theorem \ref{mt},
the graph of the solution $u_{\infty}$ is part of a $C^{1,1}$ manifold with
boundary $\partial\Omega.$ However, we do not know if $u_{\infty}\in
C^{0}\left(  \overline{\Omega}\right)  $. The $2-$dimensional case is studied
in \cite{RT} using classical Plateau's problem technique which is typically
$2-$dimensional.

\begin{figure}[h]
\centering
\includegraphics[width=1\linewidth]{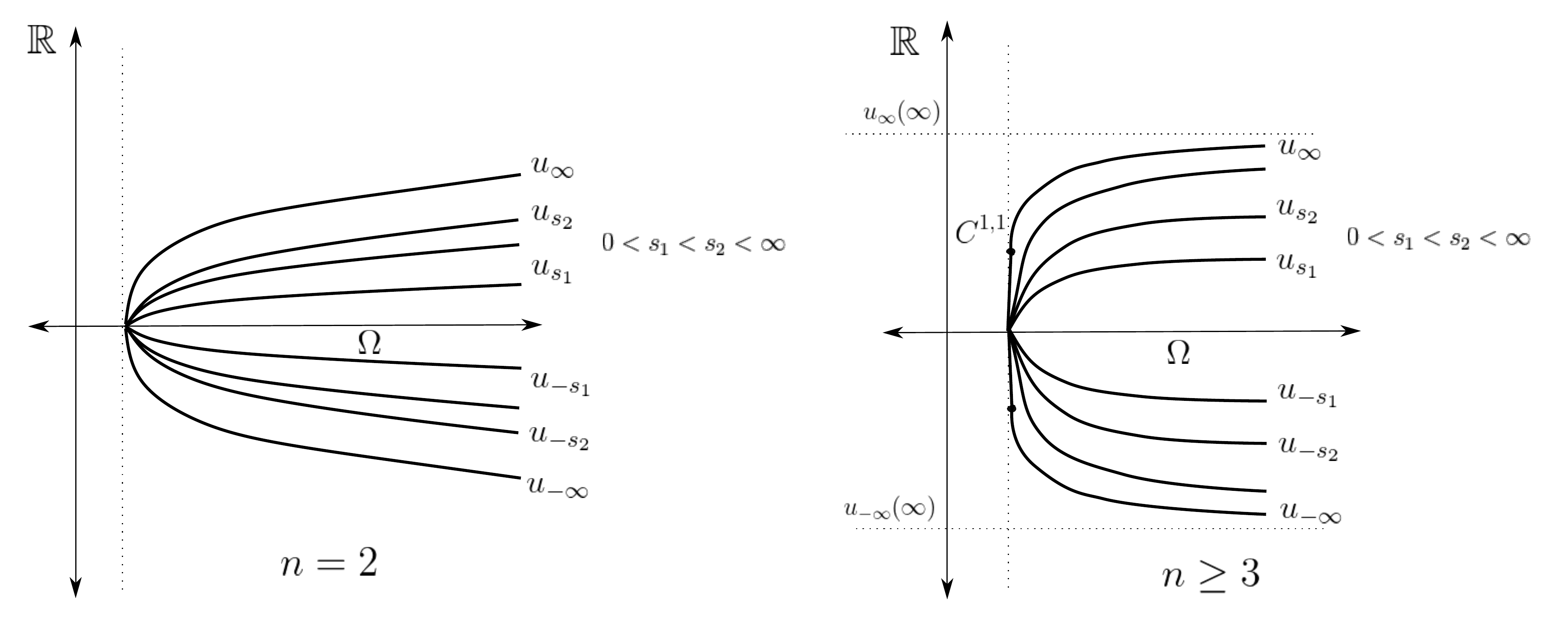}
\caption*{Possible solutions of arbitrary domains }
\label{fig:fig2}%
\end{figure}


\noindent(b) The EDP for the minimal surface equation is studied in the
Riemannian setting in \cite{ARS} and \cite{ER}.

\medskip

\noindent Ari Aiolfi \newline\noindent Universidade Federal de Santa
Maria\newline\noindent ari.aiolfi@ufsm.br \newline\noindent Brazil \newline

\medskip

\noindent Daniel Bustos \newline\noindent Universidad del Tolima
\newline\noindent dfbustosr@ut.edu.co \newline\noindent Colombia \newline

\medskip

\noindent Jaime Ripoll\newline\noindent Universidade Federal do Rio Grande do
Sul\newline\noindent jaime.ripoll@ufrgs.br\newline\noindent Brazil

\end{document}